\documentclass[11pt]{amsart}

\usepackage{amscd, stmaryrd}
\usepackage{amsmath, amssymb}
\usepackage{amsfonts}
\newcommand{\de}{\partial}

\newcommand{\ddbar}{\sqrt{-1} \partial \overline{\partial}}
\newcommand{\Ric}{\mathrm{Ric}}
\newcommand{\ov}[1]{\overline{#1}}

\newcommand{\tr}[2]{\textrm{tr}_{#1}{#2}}
\newcommand{\ti}[1]{\tilde{#1}}
\newcommand{\vp}{\varphi}

\newcommand{\ve}{\varepsilon}

\renewcommand{\leq}{\leqslant}
\renewcommand{\geq}{\geqslant}
\renewcommand{\le}{\leqslant}
\renewcommand{\ge}{\geqslant}

\newcommand{\be}{\begin{equation}}
\newcommand{\ee}{\end{equation}}

\newcommand{\fvebe}{f_{\ve,\beta}}
\newcommand{\tifvebe}{\tilde{f}_{\ve,\beta}}
\begin{document}
\newtheorem{claim}{Claim}
\newtheorem{theorem}{Theorem}[section]
\newtheorem{lemma}[theorem]{Lemma}
\newtheorem{corollary}[theorem]{Corollary}
\newtheorem{proposition}[theorem]{Proposition}
\newtheorem{question}{question}[section]
\newtheorem{defn}{Definition}[theorem]
\theoremstyle{definition}
\newtheorem{remark}[theorem]{Remark}

\numberwithin{equation}{section}

\title[Degenerate complex Monge-Amp\`ere equations]{$C^{1,1}$ regularity for degenerate complex Monge-Amp\`ere equations and geodesic rays}
\author[J. Chu]{Jianchun Chu}
\address{Institute of Mathematics, Academy of Mathematics and Systems Science, Chinese Academy of Sciences, Beijing 100190, P. R. China}
\author[V. Tosatti]{Valentino Tosatti}
\address{Department of Mathematics, Northwestern University, 2033 Sheridan Road, Evanston, IL 60208}
\author[B. Weinkove]{Ben Weinkove}
\address{Department of Mathematics, Northwestern University, 2033 Sheridan Road, Evanston, IL 60208}

\begin{abstract}We prove a $C^{1,1}$ estimate for solutions of complex Monge-Amp\`ere equations on compact K\"ahler manifolds with possibly nonempty boundary, in a degenerate cohomology class. This strengthens previous estimates of Phong-Sturm.  As applications we deduce the local $C^{1,1}$ regularity of geodesic rays in the space of K\"ahler metrics associated to a test configuration, as well as the local $C^{1,1}$ regularity of quasi-psh envelopes in nef and big classes away from the non-K\"ahler locus.
\end{abstract}
\maketitle

\section{Introduction}

Let $(M^n, \omega)$ be a compact K\"ahler manifold with nonempty smooth boundary $\partial M$.
In \cite{CTW0, CTW}, the authors considered a smooth solution $\varphi \in C^{\infty}(M, \mathbb{R})$ of the complex Monge-Amp\`ere equation
\begin{equation} \label{ori}
(\omega + \ddbar \varphi)^n = e^F \omega^n, \quad \omega+ \ddbar \varphi>0,
\end{equation}
and proved an interior \emph{a priori} estimate on the real Hessian $\nabla^2 \varphi$ which is independent of the infimum of the function $F$ (provided $F$ satisfies certain uniform bounds on its gradient and Hessian).  This established the existence of $C^{1,1}$ solutions of the \emph{homogeneous} complex Monge-Amp\`ere equation when $\partial M$ is weakly pseudoconcave, and settled the long standing problem of
$C^{1,1}$ regularity of Chen's weak geodesics \cite{Ch} in the space of K\"ahler potentials \cite{CTW}.  In \cite{CZ, To},  further extensions of these ideas were used to prove the $C^{1,1}$ regularity of envelopes in K\"ahler classes.

In this paper, we consider the case when the complex Monge-Amp\`ere equation is \emph{degenerate}, in the sense that the reference K\"ahler metric $\omega$ in (\ref{ori}) is replaced by a degenerate (not strictly positive) $(1,1)$ form.   This was investigated by Phong-Sturm \cite{PS3} who established $C^{1,\alpha}_{\textrm{loc}}$ estimates, for $0<\alpha<1$, under natural assumptions on $(M, \omega)$ which arise in the setting of  \emph{geodesic rays} in the space of K\"ahler metrics.  Degenerate complex Monge-Amp\`ere equations also appear in the consideration of envelopes in nef and big cohomology classes.  The main point of this paper is an improvement from $C^{1,\alpha}$  to $C^{1,1}$ regularity.

More precisely, our setup is as follows.  Let $M^n$ be a compact K\"ahler manifold with nonempty smooth boundary $\partial M$.
 Let $\omega_0$ be a smooth closed semipositive definite real $(1,1)$ form on $M$.  In addition, we assume there exists an effective divisor $E$ on $M$, disjoint from $\partial M$, together with a defining section $s \in H^0(M, \mathcal{O}(E))$ (where $\mathcal{O}(E)$ is the line bundle associated to $E$) and $h$ a Hermitian metric on $\mathcal{O}(E)$ with curvature form $R_h = - \ddbar \log h$ such that for  all sufficiently small $\delta>0$,
$$\omega_{\delta}:= \omega_0 - \delta R_h$$
is a K\"ahler form on $M$.

These hypotheses are satisfied for example when there is a modification $\mu:M\to N$ onto a compact K\"ahler $n$-manifold with boundary, where $\mu$ is given by composition of blowups with smooth centers which are compact complex submanifolds of the interior of $N$, and we take $\omega_0=\mu^*\omega_N$ for some K\"ahler metric $\omega_N$ on $N$, and $E=\mathrm{Exc}(\mu)$ is the exceptional locus of $\mu$.

Our first result is a $C^{1,1}$ regularity theorem for solutions of the homogeneous complex Monge-Amp\`ere equation with degenerate reference form $\omega_0$. It is an extension of our earlier result \cite[Corollary 1.3]{CTW} to our setting where $\omega_0$ is degenerate.

\begin{theorem} \label{mainthm}  With $M,E, \omega_0$ as above, assume in addition that $\partial M$ is weakly pseudoconcave. Let $\vp_0$ be a smooth function on $M$ with $\omega_0+\ddbar\vp_0\geq 0$, and let $\varphi$ be the unique bounded
function in $\emph{PSH}(M, \omega_0)$ solving the homogeneous complex Monge-Amp\`ere equation
\begin{equation} \label{HCMA}
(\omega_0 + \ddbar \varphi)^n = 0, \quad \textrm{on } M, \quad \varphi= \vp_0, \ \textrm{on } \partial M,
\end{equation}
in the sense of pluripotential theory.  Then $\varphi$ lies in   $C^{1,1}_{\emph{loc}}(M\setminus E)$.
\end{theorem}

The existence of a bounded $\omega_0$-plurisubharmonic $\varphi$ solving (\ref{HCMA}) follows easily from an envelope construction (see section \ref{sectiondeg}) and doesn't even require any pseudoconcavity assumption on the boundary of $M$. Phong-Sturm \cite{PS3} proved in addition that $\varphi \in C^{1, \alpha}_{\rm loc}(M \setminus E)$ for every $0<\alpha<1$, when $\de M$ is Levi-flat (and it was observed in \cite{Bo} that the argument extends to the weakly pseudoconcave case).  Our result, which makes use of the Phong-Sturm estimates, improves this regularity to $C^{1,1}_{\rm loc}(M\backslash E)$, which is optimal by the toric examples in \cite{Be3, CTa, SZ}.

As an application of Theorem \ref{mainthm} we obtain the optimal regularity result for geodesic rays constructed from test configurations.  We prove our result in the general setting of relative K\"ahler test configurations (introduced by Sj\"ostr\"om Dyrefelt \cite{Dy} and Dervan-Ross \cite{DR} independently) which contains the usual projective test configurations of Donaldson \cite{D} as a special case.  Our result, expressed in the  terminology of \cite{Dy}, is as follows:

\begin{theorem}\label{thmgeo}
Let $(X,\omega)$ be a compact K\"ahler manifold without boundary, $(\mathfrak{X},[\Omega])$ a cohomological relatively K\"ahler test configuration for $(X,[\omega])$. Given any smooth K\"ahler potential $\vp_0$ for $\omega$, let $(\vp_t)_{t\geq 0}$ be the weak geodesic ray associated to $(\mathfrak{X},[\Omega])$ emanating from $\vp_0$ (constructed in Section \ref{sectgeo}), and $\Phi$ the corresponding solution of the homogeneous complex Monge-Amp\`ere equation on $X\times \Delta^*$, for $\Delta^*$ a punctured disc in $\mathbb{C}$. Then $\Phi\in C^{1,1}_{\rm loc}(X\times\Delta^*)$.
\end{theorem}

Also, using the results of \cite{Dy}, we obtain the asymptotic behavior of the Mabuchi energy along suitable approximations $\vp_{\gamma,t}$ of the ray $\vp_t$, see \eqref{asympt} below.
Our theorem builds on and improves the result of Phong-Sturm \cite{PS3} who established $C^{1,\alpha}_{\rm loc}$ regularity for $0<\alpha<1$ in the setting of algebraic test configurations.
For background material and further references on geodesic rays, their regularity and relation to test configuration and K-stability and we refer the reader to \cite{Be,Be2,BBJ,BHJ,BHJ2,CTa,Da,DH, DaR2,Dy,PS,PS2,PS3,PS4,RWN,RWN2,RWN3,RWN4,SZ}.

Next we prove a $C^{1,1}$ regularity result for envelopes in nef and big cohomology classes on K\"ahler manifolds.  More precisely, let
 $(M^n,g)$ be a compact K\"ahler manifold without boundary and $\alpha$ a closed real $(1,1)$ form such that $[\alpha]$ is nef and  $\int_M\alpha^n>0$ (this implies that $[\alpha]$ is big, i.e. it contains a K\"ahler current, thanks to a result of Demailly-P\u{a}un \cite{DP}).   We consider the envelope
\begin{equation}\label{env}
u(x)=\sup\{\vp(x)\ |\ \vp\in \textrm{PSH}(M,\alpha), \vp\leq 0\}.
\end{equation}
Recall that there is a proper Zariski closed subset $E_{nK}(\alpha)\subset M$, the non-K\"ahler locus of $[\alpha]$, so that we can find a K\"ahler current $T=\alpha+\ddbar\psi$ with analytic singularities along $E_{nK}(\alpha)$, such that $T\geq \delta\omega$ weakly on $M$, for some $\delta>0$ (and $E_{nK}(\alpha)$ is the smallest set with this property). See e.g. \cite{CT,Dem} for background on this.

Our result is:
\begin{theorem} \label{thmenv} Let $(M,\omega)$ be a compact K\"ahler manifold and $[\alpha]$ a nef and big $(1,1)$ class. Then the envelope $u$ defined by \eqref{env} satisfies
 $u\in C^{1,1}_{\rm loc}(M\backslash E_{nK}(\alpha))$, and we have
 \begin{equation}\label{volform}
 \int_M\alpha^n=\int_{\{u=0\}}\alpha^n.
 \end{equation}
\end{theorem}

Berman \cite[Theorem 1.1]{Be} and Berman-Demailly \cite[Theorem 1.4]{BD} proved $C^{1,\gamma}_{\rm loc}$ regularity on $M \setminus E_{nK}(\alpha)$ for all $0<\gamma <1$, and earlier Berman \cite{Be3} proved Theorem \ref{thmenv} when $[\alpha]=c_1(L)$ for some big holomorphic line bundle $L$ (not necessarily nef) on a projective manifold $X$.  Theorem \ref{thmenv} was recently obtained in \cite{CZ, To} in the case when $[\alpha]$ is a K\"ahler class, in which case one obtains $u \in C^{1,1}(M)$. The equality \eqref{volform} was proved for general big $(1,1)$ classes in \cite[Theorem 1.4]{BD} (with the LHS replaced by $\mathrm{Vol}(\alpha)$ when $[\alpha]$ is not nef), with a new proof of the inequality ``$\leq$'' in \cite{Be}. Here (as in \cite{To}) we just remark that once we know $C^{1,1}$ regularity of $u$ on compact sets away from $E_{nK}(\alpha)$, then the proof of \eqref{volform} is quite easy.
Lastly, we remark that the proof of Theorem \ref{thmenv} also shows that the conclusion of the Main Theorem 1.2 of \cite{Be} is now improved to $C^{1,1}_{\rm loc}$ on the complement of the non-K\"ahler locus.\\

\noindent{\bf Acknowledgments. }The authors thank S. Boucksom and Z. Sj\"ostr\"om Dyrefelt for discussions about geodesic rays. The first-named author would like to thank his advisor G. Tian for encouragement and support. The second-named author was partially supported by NSF grant DMS-1610278, and the third-named author by NSF grant DMS-1406164.

\section{Proof of Theorem \ref{mainthm}} \label{sectiondeg}

Fix once and for all a small constant $\delta>0$ with $\omega_{\delta} = \omega_0 - \delta R_h >0$.   For ease of notation we will drop the $\delta$ subscript, writing
$$\omega = \omega_{\delta}, \quad \textrm{with } \omega = \sqrt{-1} g_{i\ov{j}} dz^i \wedge d\ov{z}^j,$$ for  $g$ the associated K\"ahler metric.

It is well-known that the existence of a bounded  $\omega_0$-plurisubharmonic $\varphi$ solving (\ref{HCMA}) follows from the envelope construction
$$\vp(x)=\sup\{u(x)\ |\ u\in PSH(M,\omega_0), \limsup_{z\to z_0}u(z)\leq \vp_0(z_0) \textrm{ for all }z_0\in \de M\},$$
see e.g. \cite[Proposition 2.7]{Be2} or \cite{RWN4}.

To prove regularity of $\vp$ we begin with an approximation argument, following Phong-Sturm \cite{PS3}.   Then define for each $\gamma \in [0,1/2]$ a reference $(1,1)$ form
$$\omega^{(\gamma)} = (1- \gamma) \omega_0 + \gamma \omega=\omega_0-\gamma\delta R_h,$$
which is K\"ahler for $0 <\gamma \le 1/2$.
Now for $0<\gamma \le 1/2$, let $\varphi_{\gamma} \in C^{\infty}(M)$ solve the non-degenerate Dirichlet problem
\begin{equation} \label{Dp}
\begin{split}
\left( \omega^{(\gamma)} + \ddbar \varphi_{\gamma} \right)^n = {} & f_{\gamma} (\omega^{(\gamma)})^n, \quad
  \omega^{(\gamma)} + \ddbar \varphi_{\gamma}>  0, \\  \varphi_{\gamma}|_{\partial M} = {} & (1-\gamma)\vp_0,
 \end{split}
\end{equation}
where we define $$f_{\gamma} = \gamma^n \frac{\omega^n}{(\omega^{(\gamma)})^n}.$$
Indeed, the function $(1-\gamma)\vp_0$ is a subsolution for \eqref{Dp} since
$$ (\omega^{(\gamma)} + (1-\gamma)\ddbar\vp_0)^n\geq \gamma^n\omega^n=f_{\gamma} (\omega^{(\gamma)})^n,$$
using that $(1-\gamma)(\omega_0+\ddbar\vp_0)\geq 0$, and hence we can apply \cite[Theorem B]{Bo} (see also \cite[Theorem 1.3]{Bl})
 to see that (\ref{Dp}) has a smooth solution $\varphi_{\gamma}$.

We wish to prove uniform estimates for $\varphi_{\gamma}$ as $\gamma \rightarrow 0$ so that $\varphi_{\gamma}$ converges to the desired solution $\varphi$ in the statement of Theorem \ref{mainthm}.  We prove our estimates in a slightly more general setting, which gives an extension of our earlier result \cite[Theorem 1.2]{CTW} to our degenerate setting.

\begin{theorem} \label{thmap} With the notation above, let $\varphi_{\gamma} \in C^{\infty}(M)$ solve
\begin{equation} \label{ma}
\begin{split}
\left( \omega^{(\gamma)} + \ddbar \varphi_{\gamma} \right)^n = {} & e^{F} \omega^n, \quad
  \omega^{(\gamma)} + \ddbar \varphi_{\gamma}>  0,
 \end{split}
\end{equation}
for a smooth function $F=F_{\gamma}$ on $M$.

Then there exist constants $B, C$ depending only on $M$, $\omega$, $\omega_0$, $E$, $s$, $h$, $\delta$, upper bounds on $\sup_M |\varphi_{\gamma}| $, $\sup_{\partial M} |\de \varphi_{\gamma}|_g$, $\sup_{\partial M} |\nabla^2 \varphi_{\gamma}|_g$, $\sup_M F$, $\sup_M |\de F|_g$ and on a lower bound of $\nabla^2 F$ with respect to $g$ such that on $M\backslash E$ we have
\begin{equation} \label{eqap}
 |\nabla^2 \varphi_{\gamma}|_g \le \frac{C}{|s|_h^{2B}}.
\end{equation}
\end{theorem}

Crucially, the constants $B, C$ do not depend on $\inf_M F$ and are independent of $\gamma$ (if the above bounds for $F=F_\gamma$ are also independent of $\gamma$).  When $\omega^{(\gamma)}$ is replaced by a fixed K\"ahler form, the estimate (\ref{eqap}) with $B=0$ was established in \cite{CTW} (as well as in \cite{CTW0} in a more general setting).  Phong-Sturm \cite[Theorem 1]{PS3} previously showed that under the same assumptions as Theorem \ref{thmap} one has the estimates
\begin{equation} \label{PSbounds}
 | \de \varphi_{\gamma}|^2_g  \le \frac{C}{|s|_h^{2B}}, \quad  | \Delta_g \varphi_{\gamma}|  \le \frac{C}{|s|_h^{2B}},
\end{equation}
on $M\backslash E$, for $B, C$ depending on the same quantities as described in Theorem \ref{thmap} (in fact the dependence can be weakened slightly in the obvious way, replacing the bounds on $\sup_{\partial M} |\nabla^2 \varphi_{\gamma}|_g$ and $\nabla^2 F$ with the appropriate Laplacian bounds).  We will make use of the Phong-Sturm estimates (\ref{PSbounds}) in our proof.
Note also that Theorem \ref{thmap} and its proof remain valid also when $\de M=\emptyset$.

Assuming Theorem \ref{thmap} for the moment, we complete the proof of Theorem \ref{mainthm}.  Let $\varphi_{\gamma}$ solve (\ref{Dp}).  Then it is known from \cite{Bl,CKNS,Ch,Gu} (see also the  expositions in \cite{Bo, PS3}) that we have uniform bounds on  $\sup_M |\varphi_{\gamma}| $, $\sup_{\partial M} |\de \varphi_{\gamma}|_g$ and $\sup_{\partial M} |\nabla^2 \varphi_{\gamma}|_g$.  Define $F=F_{\gamma}$ to be the constant $$F = n \log \gamma$$
 so that we trivially have upper bounds on $\sup_M F$, $\sup_M |\de F|_g$ and a lower bound of $\nabla^2F$.  We can now apply Theorem \ref{thmap} to $\varphi_{\gamma}$ and
 letting $\gamma \rightarrow 0$ we obtain our solution $\varphi \in C^{1,1}_{\textrm{loc}}(M \setminus E)$ of (\ref{HCMA}) as required (the fact that $\vp$ solves \eqref{HCMA} follows easily from the fact that $f_\gamma\to 0$ uniformly, together with the Chern-Levine-Nirenberg inequality, as in \cite{PS3}). This of course coincides with the solution $\vp$ defined as an envelope, by uniqueness.

It remains to prove the \emph{a priori} estimates.

\begin{proof}[Proof of Theorem \ref{thmap}]  Define
$$\varphi_{\gamma,\delta} = \varphi_{\gamma} - (1-\gamma)\delta \log |s|^2_h$$
so that on $M \setminus E$, recalling that we are writing $\omega$ for $\omega_{\delta}$,
\begin{equation} \label{to}
\tilde{\omega} : = \omega^{(\gamma)} + \ddbar \varphi_{\gamma} = \omega+ \ddbar  \varphi_{\gamma, \delta}>0.
\end{equation}
We will write $\tilde{g}_{i\ov{j}}$ for the corresponding K\"ahler metric.  Then the equation (\ref{ma}) can be written
\begin{equation} \label{ma2}
\log \det \tilde{g} = F + \log \det g.
\end{equation}
Note that since $F$ is assumed to be bounded from above, we have from the arithmetic-geometric means inequality,
\begin{equation} \label{ag}
\tr{\ti{\omega}}{\omega} \ge c,
\end{equation}
for a uniform constant $c>0$. The second Phong-Sturm estimate in \eqref{PSbounds} implies that
\begin{equation} \label{db}
\tr{\omega}{\ti{\omega}}\leq \frac{C}{|s|_h^{2B}},
\end{equation}

Up to scaling the section $s$, we can assume without loss of generality that $|s|^2_h\leq 1$ on $M$.

Let $B$ be a uniform constant at least as large as the constant $B$ of the Phong-Sturm estimates  (\ref{PSbounds}).
We apply the maximum principle to the quantity
$$Q = \log \lambda_1( \nabla^2 \varphi_{\gamma}) + \rho(|s|_h^{2B} | \partial \varphi_{\gamma} |^2_g) - A\varphi_{\gamma,\delta},$$
where $\lambda_1( \nabla^2 \varphi_{\gamma})$ is the largest eigenvalue of the real Hessian $\nabla^2 \varphi_{\gamma}$ (with respect to the metric $g$), and this quantity is defined on the set of points where $\lambda_1(\nabla^2 \varphi_\gamma)>0$ (which we may assume is nonempty).
The function $\rho$ is given by
\begin{equation} \label{defnrho}
\rho(\tau) = - \frac{1}{2} \log (1+ \sup_M (|s|_h^{2B}|\partial \varphi_{\gamma}|^2_g)  - \tau),
\end{equation}
and $A>1$ is a constant to be determined (which will be uniform, in the sense that it will depend only on the background data).  Note that $\rho(|s|_h^{2B}| \partial \varphi_{\gamma}|^2_g)$ is uniformly bounded thanks to \eqref{PSbounds}, and
\begin{equation} \label{proph}
\frac{1}{2}\geq \rho' \geq \frac{1}{2+2\sup_M( |s|_h^{2B}|\partial \varphi_{\gamma}|^2_g)}>0, \quad \textrm{and } \rho'' = 2 (\rho')^2,
\end{equation}
where we are evaluating $\rho$ and its derivatives at $|s|_h^{2B}|\de\vp_{\gamma}|^2_g$.

  Note that the maximum of $Q$ is achieved at a point $x_0$ of $M\backslash E$ where $\lambda_1(\nabla^2 \varphi_\gamma)(x_0)>0$, and we may also assume without loss of generality that it is not attained on $\partial M$.  The goal is to prove that at $x_0$ we have
  \begin{equation} \label{goalold}
\lambda_1(\nabla^2 \varphi_{\gamma}) \le C|s|_h^{-4B},
\end{equation}
where here and in the rest of the paper $C$ denotes a positive constant which is uniform, in the sense that it depends only on the allowable background data, and which may change from line to line.
Indeed, if we have this, then at $x_0$ we have
\begin{equation} \label{upperQ}
Q\leq C-2B\log|s|^2_h+\frac{A\delta}{2}\log|s|^2_h\leq C,
\end{equation}
as long as $A\delta \ge 4B$ (recall that $|s|^2_h\leq 1$ and $\gamma \le 1/2$).
Hence $Q\leq C$ holds everywhere, which then implies that
$$\sup_M |s|^{A\delta}_h|\nabla^2\vp_{\gamma}|_g \leq C,$$
as desired.

We apply a perturbation argument, as in \cite{S, STW, CTW0, CTW}, to avoid the situation when  the eigenspace of $\lambda_1$ has dimension greater than $1$.  Pick  holomorphic normal coordinates  centered at $x_0$ such that $(g_{i\ov{j}})$ is the identity and $\tilde{g}_{i\ov{j}}$ is diagonal with $\tilde{g}_{1\ov{1}} \ge \cdots \ge \tilde{g}_{n\ov{n}}$ at $x_0$.  We use the same notation as in \cite{CTW}, in particular using Greek letters for the ``real'' indices ranging from $1$ to $2n$ in contrast to Roman letters for the ``complex'' indices.
Write $V_1$ for a unit eigenvector for $\nabla^2\varphi_{\gamma}$ at $x_0$ corresponding to $\lambda_1$ and extend to an orthonormal basis $V_1, \ldots, V_{2n}$ of eigenvectors at $x_0$ with eigenvalues $\lambda_1 \ge \cdots \ge \lambda_{2n}$.  We extend the $V_{\beta}$ to vector fields in a neighborhood of $x_0$ with constant coefficients $(V_\beta^1,\dots V_\beta^{2n})$ in our coordinates.

Define a local endomorphism $\Phi^{\alpha}_{\ \beta}$ of $TM$ by $\Phi^{\alpha}_{\ \beta}: = g^{\alpha \sigma} \nabla^2_{\sigma \beta} \varphi_{\gamma} - g^{\alpha \sigma} B_{\sigma \beta}$ where $B_{\alpha \beta} : = \delta_{\alpha \beta} - V_1^{\alpha} V_1^{\beta}$ is semi-positive definite.  The $V_{\alpha}$ are eigenvectors for $\Phi$ at $x_0$ with eigenvalues $\lambda_1(\Phi) > \lambda_2(\Phi) \ge \cdots \ge \lambda_{2n} (\Phi)$.  Moreover, $\lambda_1(\Phi) \le \lambda_1(\nabla^2 \varphi_{\gamma})$ near $x_0$ with equality at $x_0$.
We then define
$$\hat{Q} =  \log \lambda_1(\Phi) + \rho(|s|_h^{2B} | \partial \varphi_{\gamma} |^2_g) - A\varphi_{\gamma,\delta}.$$
Note that $\hat{Q}$ still attains a maximum at $x_0$, and by the same argument above it suffices to show that at $x_0$ we have
\begin{equation} \label{goal}
\lambda_1 \le C |s|_h^{-4B},
\end{equation}
where for convenience we write $\lambda_{\alpha}$ for $\lambda_{\alpha}(\Phi)$.

First we have a lemma.

\begin{lemma} \label{lemmafo} Writing $\Delta_{\tilde{g}} = \tilde{g}^{i\ov{j}} \partial_i \partial_{\ov{j}}$ we have at $x_0$,
\begin{equation}\label{equJ}
\begin{split}
\Delta_{\tilde{g}}(\rho(|s|_{h}^{2B}|\de\vp_{\gamma}|_{g}^{2}))
\geq & ~ \frac{\rho'}{2}|s|_{h}^{2B}\sum_{k}\tilde{g}^{i\overline{i}}(|(\varphi_{\gamma})_{ik}|^{2}+|(\vp_{\gamma})_{i\ov{k}}|^{2})
\\ & +\rho''\tilde{g}^{i\overline{i}}|\partial_{i}(|s|_{h}^{2B}|\de\vp_{\gamma}|_{g}^{2})|^{2}
  -C\sum_{i}\tilde{g}^{i\overline{i}}.
\end{split}
\end{equation}
\end{lemma}
\begin{proof}  We compute everything at $x_0$, using the normal coordinate system described above.
Differentiating (\ref{ma2}) we obtain
$$\tilde{g}^{i\ov{i}} ( \partial_{\ov{k}} g^{(\gamma)}_{i\ov{i}} + \partial_{\ov{k}} \partial_i \partial_{\ov{i}} \varphi_{\gamma}) = F_{\ov{k}}.$$
Hence, using (\ref{PSbounds}),
\begin{equation}
\begin{split}
\Delta_{\tilde{g}} ( | \partial \varphi_{\gamma}|^2_g) ={} & \sum_k \tilde{g}^{i\ov{i}} (| (\varphi_{\gamma})_{ik}|^2 + | (\varphi_{\gamma})_{i\ov{k}}|^2) + \tilde{g}^{i\ov{i}} \partial_i \partial_{\ov{i}} (g^{k\ov{\ell}}) (\varphi_{\gamma})_k (\varphi_{\gamma})_{\ov{\ell}} \\ {} & + 2\textrm{Re}\left( \sum_k (\varphi_{\gamma})_k (F_{\ov{k}} - \tilde{g}^{i\ov{i}} \partial_{\ov{k}} g^{(\gamma)}_{i\ov{i}} )\right)   \\
\ge {} & \sum_k \tilde{g}^{i\ov{i}} (| (\varphi_{\gamma})_{ik}|^2+ | (\varphi_{\gamma})_{i\ov{k}}|^2) - C |s|^{-2B}_h \sum_i \tilde{g}^{i\ov{i}},
\end{split}
\end{equation}
where we have used (\ref{ag}).
On $M\setminus E$, we have
$$\ddbar |s|_h^{2B}\geq|s|_h^{2B}\ddbar\log|s|^{2B}_h=-B|s|_h^{2B}R_h\geq -C|s|^{2B}_h\omega.$$
Then
\begin{equation}\label{L}
\begin{split}
\Delta_{\tilde{g}}(|s|_h^{2B}|\de\vp_{\gamma}|^2_g)={} & |s|_h^{2B}\Delta_{\tilde{g}}(|\de\vp_{\gamma}|^2_g)+|\de\vp_{\gamma}|^2_g\Delta_{\tilde{g}}(|s|_h^{2B}) \\ {} & +2\mathrm{Re}\left(\ti{g}^{i\ov{i}}\partial_i(|s|_h^{2B})\partial_{\ov{i}}(|\de\vp_{\gamma}|^2_g)\right)\\
\geq {} &  |s|_h^{2B}\sum_k \tilde{g}^{i\ov{i}} (| (\varphi_{\gamma})_{ik}|^2+ | (\varphi_{\gamma})_{i\ov{k}}|^2)  - C \sum_i \tilde{g}^{i\ov{i}}\\
&+2\mathrm{Re}\left(\ti{g}^{i\ov{i}}\partial_i(|s|_h^{2B})\partial_{\ov{i}}(|\de\vp_{\gamma}|^2_g)\right).\\
\end{split}
\end{equation}
We need to deal with the third term on the right hand side of this inequality.  Using Cauchy-Schwarz and Young's inequalities,
\[
\begin{split}
\lefteqn{2\mathrm{Re}\left(\ti{g}^{i\ov{i}}\partial_i(|s|_h^{2B})\partial_{\ov{i}}(|\de\vp_{\gamma}|^2_g)\right) } \\
= {} & 2 \textrm{Re} \left( \ti{g}^{i\ov{i}} B |s|_h^{2(B-1)} (\partial_i |s|^2_h) \sum_k \left( (\varphi_{\gamma})_{k\ov{i}} (\varphi_{\gamma})_{\ov{k}} + (\varphi_{\gamma})_k (\varphi_{\gamma})_{\ov{k} \ov{i}} \right) \right) \\
\ge {} & - \frac{|s|^{2B}_h}{2} \sum_k \tilde{g}^{i\ov{i}} (| (\varphi_{\gamma})_{ik}|^2+ | (\varphi_{\gamma})_{i\ov{k}}|^2) -  2B^2  |s|_h^{2B-4} |\partial \varphi_\gamma|^2_g \tilde{g}^{i\ov{i}} | \partial_i |s|^2_h|^2 \\
\ge {} & - \frac{|s|^{2B}_h}{2} \sum_k \tilde{g}^{i\ov{i}} (| (\varphi_{\gamma})_{ik}|^2+ | (\varphi_{\gamma})_{i\ov{k}}|^2) - C \sum_i \tilde{g}^{i\ov{i}},
\end{split}
\]
where for the last line we used the fact that $|s|^2_h$ is smooth, and we increased $B$ if necessary to ensure that $|\partial \varphi_\gamma|^2_g |s|_h^{2B-4} \le C$ by \eqref{PSbounds}.  Combining this with (\ref{L}) gives
\begin{equation}\label{ggg}
\Delta_{\tilde{g}}(|s|_h^{2B}|\de\vp_{\gamma}|^2_g)
\geq \frac{|s|_h^{2B}}{2} \sum_k \tilde{g}^{i\ov{i}} (| (\varphi_{\gamma})_{ik}|^2+ | (\varphi_{\gamma})_{i\ov{k}}|^2)  - C \sum_i \tilde{g}^{i\ov{i}},
\end{equation}
and (\ref{equJ}) follows, using the fact that $0<\rho'\le C$.
\end{proof}

We then obtain the following lower bound for $\Delta_{\ti{g}} \hat{Q}$, analogous to \cite[Lemma 2.1]{CTW}.

\begin{lemma} \label{lemmalbl1}  At $x_0$, we have
\begin{equation} \label{LhatQ4}
\begin{split}
0 \ge \Delta_{\ti{g}} \hat{Q}
\ge {} & 2 \sum_{\alpha >1}  \frac{\tilde{g}^{i\ov{i}} |\partial_i ((\varphi_{\gamma})_{V_{\alpha} V_1})|^2}{\lambda_1(\lambda_1-\lambda_{\alpha})} + \frac{\tilde{g}^{p\ov{p}} \tilde{g}^{q\ov{q}} | V_1(\tilde{g}_{p\ov{q}})|^2}{\lambda_1} - \frac{\tilde{g}^{i\ov{i}} | \partial_i ((\varphi_{\gamma})_{V_1 V_1})|^2}{\lambda_1^2} \\
{} & + |s|_h^{2B}\frac{\rho'}{2} \sum_k \tilde{g}^{i\ov{i}} (| (\varphi_{\gamma})_{ik}|^2 + | (\varphi_{\gamma})_{i\ov{k}}|^2) + \rho'' \tilde{g}^{i\ov{i}} |\partial_i (|s|_h^{2B}| \partial \varphi_{\gamma}|^2_g)|^2 \\ {} &+ \left(A-C\right) \sum_i \tilde{g}^{i\ov{i}}
 - An,
\end{split}
\end{equation}
where
$$(\varphi_{\gamma})_{V_{\alpha} V_{\beta}} = \nabla^2 \varphi_{\gamma} (V_{\alpha}, V_{\beta}).$$
\end{lemma}
\begin{proof}
From the same proof as equation (2.8) in \cite{CTW} we have
$$\Delta_{\ti{g}} \lambda_1 \ge 2\sum_{\alpha>1} \tilde{g}^{i\ov{i}} \frac{| \partial_i ((\varphi_{\gamma})_{V_{\alpha} V_1}) |^2}{\lambda_1-\lambda_{\alpha}} + \tilde{g}^{i\ov{i}} \partial_i \partial_{\ov{i}} (( \varphi_{\gamma})_{V_1 V_1})
 - C\lambda_1 \sum_i \tilde{g}^{i\ov{i}}.$$
 Compute
 \[
 \begin{split}
 \tilde{g}^{i\ov{i}} \partial_i \partial_{\ov{i}} (( \varphi_{\gamma})_{V_1 V_1}) \ge {} & \tilde{g}^{i\ov{i}} V_1 V_1 (\partial_i \partial_{\ov{i}} \varphi_{\gamma}) - C(| \partial \varphi_{\gamma}|_g + \lambda_1) \sum_i \tilde{g}^{i\ov{i}} \\
 \ge {} & \tilde{g}^{i\ov{i}} V_1 V_1 (\tilde{g}_{i\ov{i}}) -  C\lambda_1 \sum_i \tilde{g}^{i\ov{i}}, \\
 \end{split}
\]
where for the last inequality we assumed, without loss of generality (due to the bound (\ref{PSbounds}) and our goal (\ref{goal})), that $|\partial \varphi_{\gamma}|_g \le C \lambda_1$ at $x_0$.
Applying $V_1V_1$ to (\ref{ma2}) we obtain
\[
\begin{split}
\tilde{g}^{i\ov{i}} V_1 V_1 (\tilde{g}_{i\ov{i}}) = {} & \tilde{g}^{p\ov{p}} \tilde{g}^{q\ov{q}} |V_1(\tilde{g}_{p\ov{q}})|^2  + V_1 V_1(F)+ V_1V_1(\log \det g) \\
\ge {} & \tilde{g}^{p\ov{p}} \tilde{g}^{q\ov{q}} |V_1(\tilde{g}_{p\ov{q}})|^2 - C.
\end{split}
\]
 Combining the above gives
$$\Delta_{\ti{g}} \lambda_1 \ge 2\sum_{\alpha>1} \tilde{g}^{i\ov{i}} \frac{| \partial_i ((\varphi_{\gamma})_{V_{\alpha} V_1}) |^2}{\lambda_1-\lambda_{\alpha}} + \tilde{g}^{p\ov{p}} \tilde{g}^{q\ov{q}} |V_1 (\tilde{g}_{p\ov{q}})|^2 - C\lambda_1 \sum_i \tilde{g}^{i\ov{i}},$$
and hence
\[
\begin{split}
\Delta_{\ti{g}} \log \lambda_1 \ge {} & 2\sum_{\alpha>1} \tilde{g}^{i\ov{i}} \frac{| \partial_i ((\varphi_{\gamma})_{V_{\alpha} V_1}) |^2}{\lambda_1(\lambda_1-\lambda_{\alpha})} + \frac{\tilde{g}^{p\ov{p}} \tilde{g}^{q\ov{q}} |V_1 (\tilde{g}_{p\ov{q}})|^2}{\lambda_1} \\ {} & - \frac{\tilde{g}^{i\ov{i}} | \partial_i ((\varphi_{\gamma})_{V_1 V_1})|^2}{\lambda_1^2}
  - C  \sum_i \tilde{g}^{i\ov{i}}.
\end{split}
\]
From (\ref{to}) we have
$$- \Delta_{\ti{g}} (A \varphi_{\gamma, \delta}) = A\tilde{g}^{i\ov{i}} ( g_{i\ov{i}} - \ti{g}_{i\ov{i}}) = A \sum_i \tilde{g}^{i\ov{i}} - An.$$
Applying Lemma \ref{lemmafo}, the inequality (\ref{LhatQ4}) follows.
\end{proof}

Next we have an analog of \cite[Lemma 2.2]{CTW}.

\begin{lemma}\label{lemmauno} There is a uniform constant $C\geq 1$ such that if $0 < \ve< 1/2$ and $\lambda_1(x_0) \ge C/\ve^2,$ then at $x_0$ we have
\begin{equation}
\begin{split}
\sum_i \frac{\tilde{g}^{i\ov{i}} |\partial_i ((\varphi_{\gamma})_{V_1V_1})|^2}{\lambda_1^2} \le {} & 2(\rho')^2 \tilde{g}^{i\ov{i}} |\partial_i (|s|_h^{2B}|\partial \varphi_{\gamma}|_g^2)|^2 + 4 \ve A^{2} \tilde{g}^{i\ov{i}}|(\varphi_{\gamma, \delta})_i|^2 \\
{} & +2\sum_{\alpha>1} \frac{\tilde{g}^{i\ov{i}} | \partial_i ((\varphi_{\gamma})_{V_{\alpha} V_1})|^2}{\lambda_1(\lambda_1-\lambda_{\alpha})} + \frac{\tilde{g}^{p\ov{p}} \tilde{g}^{q\ov{q}} |V_1(\tilde{g}_{p\ov{q}})|^2}{\lambda_1} +  \sum_i \tilde{g}^{i\ov{i}}.
\end{split}
\end{equation}
\end{lemma}
\begin{proof}
We omit the proof since it is almost identical to that of  \cite[Lemma 2.2]{CTW}.  There are only two small differences. First, the error term $E$ there contains terms of order $O(|\partial \varphi_{\gamma}|_g)$ and instead of $|E|\le C$ we have  $|E| \le C(\lambda_1)^{1/4}$ using (\ref{PSbounds}).  This change does not affect the rest of the proof. Second, since \eqref{db} here is weaker than \cite[(2.5)]{CTW}, we replace \cite[(2.21)]{CTW} by
\begin{equation} \label{tech}
\begin{split}
(1-\ve)\left(1+\frac{1}{\ve}\right)\sum_{i}\frac{2\tilde{g}^{i\ov{i}}}{\lambda_1^2} \left|\sum_{q}\ov{\nu_q}V_1(\ti{g}_{i\ov{q}})\right|^2 \le {} & \frac{C}{\ve |s|^{2B}_h\lambda_1} \sum_{i} \sum_{q} \frac{\tilde{g}^{i\ov{i}} \tilde{g}^{q\ov{q}}  |V_1 (\tilde{g}_{i\ov{q}})|^2}{\lambda_1} \\
\le {} &   \sum_{i} \sum_{q} \frac{\tilde{g}^{i\ov{i}} \tilde{g}^{q\ov{q}} |V_1 (\tilde{g}_{i\ov{q}})|^2}{\lambda_1},
\end{split}
\end{equation}
using that without loss of generality we may assume that $\lambda_1\geq C|s|_h^{-4B},$ which together with $\lambda_1\geq C/\ve^2$ gives $\frac{C}{\ve |s|^{2B}_h\lambda_1}\leq 1$. Now the rest of the proof goes through unchanged.
\end{proof}

We continue to prove the second order estimate. Using the inequality
\begin{equation*}
|(\varphi_{\gamma,\delta})_i|^2 \leq \frac{C}{|s|^{2B}_h},
\end{equation*}
which comes from \eqref{PSbounds} and the definition of $\vp_{\gamma,\delta}$, together with the fact that $\rho'' = 2 (\rho')^2$,
we combine Lemma \ref{lemmalbl1} and Lemma \ref{lemmauno} to obtain, increasing the uniform constant $C$ if necessary,
\begin{equation}\label{satan}
\begin{split}
0 \ge &\left(A-C\right)\sum_{i}\tilde{g}^{i\ov{i}}+\frac{|s|_h^{2B}}{2}\rho' \sum_k \tilde{g}^{i\ov{i}} (| (\varphi_{\gamma})_{ik}|^2 + | (\varphi_{\gamma})_{i\ov{k}}|^2)\\
&-4 \ve A^2 \tilde{g}^{i\ov{i}}|(\varphi_{\gamma, \delta})_i|^2 - An\\
\ge &\left(A-C-\frac{4C \ve A^2}{|s|_h^{2B}}\right)\sum_{i}\tilde{g}^{i\ov{i}}+\frac{|s|_h^{2B}}{2}\rho' \sum_k \tilde{g}^{i\ov{i}} (| (\varphi_{\gamma})_{ik}|^2 + |(\varphi_{\gamma})_{i\ov{k}}|^2)- An,
\end{split}
\end{equation}
as long as $\ve<\frac{1}{2}$ and $\lambda_1 \ge C/\ve^2$. We choose $A=\max(3C, 4B/\delta)$ and $\ve=\frac{|s|_h^{2B}(x_0)}{4A^2}$, where we recall that we needed $A \ge 4B/\delta$ for (\ref{upperQ}) above.
If at $x_0$ we have $\lambda_1 \leq C/\ve^2$, we get $\lambda_1\leq C|s|^{-4B}_h$, which is (\ref{goal}) and we are done.

Otherwise, $\lambda_1 \ge C/\ve^2$, then from \eqref{satan} we get
\begin{equation}\label{temp}
\begin{split}
0\geq {} & \sum_{i}\tilde{g}^{i\ov{i}}+\frac{|s|_h^{2B}}{2}\rho' \sum_k \tilde{g}^{i\ov{i}} (| (\varphi_{\gamma})_{ik}|^2 + | (\varphi_{\gamma})_{i\ov{k}}|^2)- An.
\end{split}
\end{equation}
This gives $\sum_{i}\tilde{g}^{i\ov{i}}\leq C$ at $x_0$, and using the elementary inequality
$$\sum_{i}\tilde{g}_{i\ov{i}}\leq \frac{1}{(n-1)!}\left(\sum_{j}\tilde{g}^{j\ov{j}}\right)^{n-1}\prod_k \ti{g}_{k\ov{k}},$$
together with the Monge-Amp\`ere equation \eqref{ma}, we see that $\ti{g}$ and $g$ are uniformly equivalent at $x_0$. From \eqref{temp} we then obtain
$$\frac{|s|_h^{2B}}{C} \sum_{i,k} (| (\varphi_{\gamma})_{ik}|^2 + | (\varphi_{\gamma})_{i\ov{k}}|^2)\leq C,$$
where we also used $\rho' \ge C^{-1}$ which follows from (\ref{proph}) and the first estimate of (\ref{PSbounds}).
By the definition of $\lambda_1$  we then  have $\lambda_1\leq C|s|_h^{-2B}$ at $x_0$, which again gives (\ref{goal}), as required.
 This completes the proof.
\end{proof}

\section{Geodesic rays from test configurations}\label{sectgeo}
In this section we explain how to derive Theorem \ref{thmgeo} from Theorem \ref{mainthm}.\\

Let $(X^n,\omega)$ be a compact K\"ahler manifold without boundary, and let $(\mathfrak{X},[\Omega])$ be a cohomological relatively K\"ahler test configuration for $(X,[\omega])$ in the sense of \cite{Dy} (see also \cite{DR}). The total space $\mathfrak{X}$ is a normal compact K\"ahler analytic space with a surjective map $\pi:\mathfrak{X}\to\mathbb{CP}^1$ with a $\mathbb{C}^*$-action on $\mathfrak{X}$ lifting the usual action on $\mathbb{CP}^1$, with a $\mathbb{C}^*$-equivariant isomorphism $\rho:\mathfrak{X}\backslash \mathfrak{X}_0\cong X\times(\mathbb{CP}^1\backslash \{0\})$, which defines a bimeromorphic map $\rho:\mathfrak{X}\dashrightarrow X\times\mathbb{CP}^1$.
The class $[\Omega]$ is a $\mathbb{C}^*$-invariant $(1,1)$ real Bott-Chern cohomology class, whose restriction to $\mathfrak{X}\backslash \mathfrak{X}_0$ is pulled back by $\rho^{-1}$ to $p_1^*[\omega]$ on $X\times(\mathbb{CP}^1\backslash\{0\})$, where $p_1:X\times\mathbb{CP}^1\to X$ is the projection, and which is $\pi$-K\"ahler in the sense that $[\Omega]+\pi^*[\eta]$ is K\"ahler on $\mathfrak{X}$ for some $S^1$-invariant K\"ahler class $[\eta]$ on $\mathbb{CP}^1$.
If we fix a smooth $S^1$-invariant representative $\Omega$ of $[\Omega]$, then by assumption there is an $S^1$-invariant K\"ahler form $\eta$ on $\mathbb{CP}^1$ such that $\Omega+\pi^*\eta$ is K\"ahler on $\mathfrak{X}.$

This definition is a transcendental generalization of Donaldson's definition \cite{D} of a test configuration $(\mathfrak{X},\mathfrak{L})$ for a polarized compact K\"ahler manifold $(X,L)$, and in particular the results of this section apply directly to such (usual) test configurations.

Using Hironaka's Theorem we obtain a $\mathbb{C}^*$-equivariant modification $\mu:\ti{\mathfrak{X}}\to\mathfrak{X}$ so that $\ti{\mathfrak{X}}$ is smooth and the indeterminacies of $\rho$ are resolved, with $\mu$ an isomorphism away from $\mathfrak{X}_0$, so that we have a bimeromorphic morphism $\rho:\ti{\mathfrak{X}}\to X\times\mathbb{CP}^1$ (using the same notation as before), and we also let $\ti{\pi}=\pi\circ\mu:\ti{\mathfrak{X}}\to\mathbb{CP}^1$. Let $\ti{\Omega}=\mu^*(\Omega+\pi^*\eta),$ a smooth semipositive form on $\ti{\mathfrak{X}}$. There is a $\mu$-exceptional effective integral divisor $E$ on $\ti{\mathfrak{X}}$, supported on $\ti{\mathfrak{X}}_0$, such that (with the same notation as earlier for $s$,$h$, etc.) for all sufficiently small $\delta>0$ we have that $\ti{\Omega}-\delta R_h$ is an $S^1$-invariant K\"ahler form on $\ti{\mathfrak{X}}$, and we have
$$\ddbar\log|s|^2_h=\llbracket E\rrbracket-R_h,$$
with $R_h$ and $\log |s|^2_h$ both $S^1$-invariant, where $\llbracket E\rrbracket$ denotes the current of integration along $E$.
Furthermore, since $(\ti{\mathfrak{X}},[\mu^*\Omega])$ is a smooth cohomological test configuration for $(X,[\omega])$, \cite[Proposition 3.8]{Dy} shows that on $\ti{\mathfrak{X}}$ we have
$$[\mu^*\Omega]=\rho^*p_1^*[\omega]+[D],$$
where $D$ as an $\mathbb{R}$-divisor on $\ti{\mathfrak{X}}$ supported on $\ti{\mathfrak{X}}_0$. By the Poincar\'e-Lelong equation there is a quasi-psh function $\psi_D$ on $\ti{\mathfrak{X}}$ (with logarithmic poles along the support of $D$ and smooth outside) which satisfies
$$\ddbar\psi_D=\llbracket D\rrbracket-R_D,$$
where $R_D$ is a smooth representative of $[D]$, and we may also assume that
$$\mu^*\Omega=\rho^*p_1^*\omega+\llbracket D\rrbracket-\ddbar\psi_D,$$
on $\ti{\mathfrak{X}}$, and that $R_D$ and $\psi_D$ are $S^1$-invariant.

Let $\Delta\subset \mathbb{CP}^1$ be the unit disc in the usual chart centered at $0\in\mathbb{CP}^1$, and denote by $\Delta^*$ the punctured disc.
 Choose a smooth $S^1$-invariant function $f$ on a neighborhood of $\Delta$ such that $\eta=\ddbar f$, and let $\vp_0$ be a given smooth function on $X$ such that $\omega+\ddbar\vp_0>0$. Let $\ti{\mathfrak{X}}_{\ov{\Delta}}$ be the preimage $\ti{\pi}^{-1}(\ov{\Delta})$, and similarly for $\ti{\mathfrak{X}}_\Delta$.
On $\ti{\mathfrak{X}}_{\ov{\Delta}}$ we have
$$\ti{\Omega}=\rho^*p_1^*\omega+\llbracket D\rrbracket+\ddbar(\ti{\pi}^*f-\psi_D),$$
and $\ti{\Omega}$ is a semipositive form which is a K\"ahler form away from $\ti{\mathfrak{X}}_0$. On $\de \ti{\mathfrak{X}}_{\ov{\Delta}}$ we let $\ti{\vp}_0=\rho^*p_1^*\vp_0+\psi_D-\ti{\pi}^*f$ (recall that $\rho$ induces an isomorphism $\de \ti{\mathfrak{X}}_{\ov{\Delta}}\cong X\times S^1$ and that $\de \ti{\mathfrak{X}}_{\ov{\Delta}}$ is Levi-flat). Note that on $\de\ti{\mathfrak{X}}_{\ov{\Delta}}$ we have
$$\ti{\Omega}+\ddbar\ti{\vp}_0=\mu^*\Omega+\ddbar(\ti{\pi}^*f+\ti{\vp}_0)=\rho^*p_1^*(\omega+\ddbar\vp_0)\geq 0,$$
and strictly positive in the $X$ directions. In order to apply Theorem \ref{mainthm} we have to find an extension of $\ti{\vp}_0$ to a smooth function $\hat{\vp}_0$ on $\ti{\mathfrak{X}}_{\ov{\Delta}}$ with $\ti{\Omega}+\ddbar\hat{\vp}_0\geq 0$. To achieve this we adapt the proof of \cite[Proposition 7.10]{Bo}: using the biholomorphism $\rho:\ti{\mathfrak{X}}_{\ov{\Delta}}\backslash \ti{\mathfrak{X}}_0\cong X\times \ov{\Delta}^*,$ we will work on this latter space. Let $U$ be an open neighborhood of $S^1\subset \ov{\Delta}$ in $\ov{\Delta}$ with a smooth retraction $r:U\to S^1$ and let $0\leq \chi\leq 1$ be a smooth function compactly supported on $U$ with $\chi\equiv 1$ in a neighborhood of $S^1$. Let $\psi(x,y)=\chi(y)\ti{\vp}_0(x,r(y))$ on $X\times U$, where $x\in X, y\in\ov{\Delta},$ and extend it to zero on all of $X\times \ov{\Delta}^*$. This function is smooth and strictly $\ti{\Omega}$-psh in the $X$ directions, for all $y$ in a neighborhood of $S^1$.
Using $\rho$ we obtain $\psi\circ\rho^{-1}$ on $\ti{\mathfrak{X}}_{\ov{\Delta}}\backslash \ti{\mathfrak{X}}_0$, which we extend by zero to all of $\ti{\mathfrak{X}}_{\ov{\Delta}}$, and we set $\hat{\vp}_0=\psi\circ\rho^{-1}+C\ti{\pi}^*u,$ where $u(z)=|z|^2-1$ on $\ov{\Delta}$, and $C$ is sufficiently large so that $\ti{\Omega}+\ddbar\hat{\vp}_0\geq 0$ on all of $\ti{\mathfrak{X}}_{\ov{\Delta}}$, as desired.

As in \cite[Section 2.4]{Be2}, we define the envelope $\Phi$ on $\ti{\mathfrak{X}}_{\ov{\Delta}}$ by
$$\Phi(x)=\sup\{u(x)\ |\ u\in PSH(\ti{\mathfrak{X}}_{\ov{\Delta}},\ti{\Omega}), \limsup_{z\to z_0}u(z)\leq \ti{\vp}_0(z_0) \textrm{ for all }z_0\in \de \mathfrak{X}_{\ov{\Delta}}\}.$$
In \cite[Proposition 2.7]{Be2} it is proved that $\Phi$ is an $S^1$-invariant locally bounded $\ti{\Omega}$-PSH function which satisfies $(\ti{\Omega}+\ddbar\Phi)^{n+1}=0$ on $\ti{\mathfrak{X}}_{\Delta}$ (in the sense of Bedford-Taylor), and the boundary values of $\Phi$ converge uniformly to $\ti{\vp}_0$.
In particular, $\ti{\Phi}=(\Phi-\psi_D+\ti{\pi}^*f)\circ\rho^{-1}$  satisfies
$$\rho^*(p_1^*\omega+\ddbar\ti{\Phi})=\ti{\Omega}+\ddbar\Phi,$$
on $\ti{\mathfrak{X}}_{\Delta}\backslash \ti{\mathfrak{X}}_{0}$, and so it
solves $(p_1^*\omega+\ddbar\ti{\Phi})^{n+1}=0$ on $X\times\Delta^*$ with boundary value $\vp_0$, and so this gives a weak geodesic ray $\vp_t,t\geq 0$ on $X$ by $\vp_t(z)=\ti{\Phi}(z,e^{-t})$, which starts at our given K\"ahler potential $\vp_0$.

Theorem \ref{mainthm} now applies directly, and shows that $\Phi\in C^{1,1}_{\rm loc}(\ti{\mathfrak{X}}_{\ov{\Delta}}\backslash \ti{\mathfrak{X}}_0)$, which translates to a $C^{1,1}_{\rm loc}(X\times\Delta^*)$ bound for the ray $\vp_t$. This proves Theorem \ref{thmgeo}.

Lastly, using the main result of \cite{Dy}, we note that the ray $\vp_t$ has the property \eqref{asympt} below, which is related to the asymptotic behavior of the Mabuchi energy (see also \cite{PS2,CTa} for earlier results along these lines). Fix $\delta>0$ sufficiently small, and for $0<\gamma\leq\frac{1}{2}$ let $\Omega^{(\gamma)}=\ti{\Omega}-\delta\gamma R_h$, which is a K\"ahler form on $\ti{\mathfrak{X}}_{\ov{\Delta}}$, which satisfies
$$[\Omega^{(\gamma)}]=\rho^*p_1^*[\omega]+[D]-\delta\gamma[E],$$
and also
$$\Omega^{(\gamma)}=\rho^*p_1^*\omega+\llbracket D\rrbracket-\delta\gamma\llbracket E\rrbracket+\ddbar(\ti{\pi}^*f-\psi_D+\delta\gamma\log|s|^2_h),$$
and the $\mathbb{R}$-divisor $D-\delta\gamma E$ is supported on $\ti{\mathfrak{X}}_0$. In particular we see that for $0<\gamma\leq \frac{1}{2}$, $(\ti{\mathfrak{X}},[\Omega^{(\gamma)}])$ is a smooth cohomological relatively K\"ahler test configuration for $(X,[\omega])$.

Let $\Phi_\gamma$ solve the analog of \eqref{Dp} in this setup, namely
\begin{equation} \label{Dp2}
\begin{split}
\left( \Omega^{(\gamma)} + \ddbar \Phi_{\gamma} \right)^n = {} & f_{\gamma} (\Omega^{(\gamma)})^n, \quad
  \Omega^{(\gamma)} + \ddbar \Phi_{\gamma}>  0, \\  \Phi_{\gamma}|_{\partial \ti{\mathfrak{X}}_{\ov{\Delta}}} = {} & (1-\gamma)\ti{\vp}_0,
 \end{split}
\end{equation}
where as before we define $$f_{\gamma} := \gamma^n \frac{(\ti{\Omega}-\delta R_h)^n}{(\Omega^{(\gamma)})^n}.$$
Then for $0<\gamma\leq\frac{1}{2}$, the function $\Phi_\gamma$ is smooth on $\ti{\mathfrak{X}}_{\ov{\Delta}}$, and as $\gamma\to 0$ we have that $\Phi_\gamma\to\Phi$ in $C^{1,1}_{\rm loc}(\ti{\mathfrak{X}}_{\ov{\Delta}}\backslash \ti{\mathfrak{X}}_0)$, and if we define $\ti{\Phi}_\gamma=(\Phi_\gamma-\psi_D+\delta\gamma\log|s|^2_h+\ti{\pi}^*f)\circ\rho^{-1},$ then
$$\rho^*(p_1^*\omega+\ddbar\ti{\Phi}_\gamma)=\Omega^{(\gamma)}+\ddbar\Phi_\gamma,$$
on $\ti{\mathfrak{X}}_{\Delta}\backslash \ti{\mathfrak{X}}_{0}$, and so it defines a ``subgeodesic ray'' $\vp_{\gamma,t}(z)=\ti{\Phi}_\gamma(z,e^{-t})$ on $X\times\Delta^*$ (i.e. $p_1^*\omega+\ddbar\ti{\Phi}_\gamma\geq 0$) which is smoothly compatible with $(\ti{\mathfrak{X}},[\Omega^{(\gamma)}])$, in the terminology of \cite{Dy}, which means that
$$\ti{\Phi}_\gamma\circ\rho+\psi_D-\delta\gamma\log|s|^2_h=\Phi_\gamma+\ti{\pi}^*f,$$
extends to a smooth function on all of $\ti{\mathfrak{X}}_{\Delta}$, which it obviously does.
We can then apply \cite[Theorem 5.5]{Dy} to get
$$\lim_{t\to\infty} M(\vp_{\gamma,t})/t=M^{\rm NA}(\ti{\mathfrak{X}},[\Omega^{(\gamma)}]),$$
where $M$ is the Mabuchi energy and $M^{\rm NA}$ is defined in \cite[Definition 3.13]{Dy},
and so
\begin{equation}\label{asympt}
\lim_{\gamma\to 0}\lim_{t\to\infty} M(\vp_{\gamma,t})/t=M^{\rm NA}(\ti{\mathfrak{X}},[\mu^*\Omega])=M^{\rm NA}(\mathfrak{X},[\Omega]).
\end{equation}
\begin{remark}
Note also that $\lim_{\gamma\to 0}\vp_{\gamma,t}=\vp_t$ in $C^{1,1}_{\rm loc}(X\times\Delta^*)$, however our estimates on $\vp_{\gamma,t}$ given by Theorem \ref{thmap} blow up rather fast as $t\to\infty$, and this does not seem sufficient to justify the exchange of the two limits above, and so the identity
\begin{equation}\label{sux}
\lim_{t\to\infty} M(\vp_{t})/t=M^{\rm NA}(\mathfrak{X},[\Omega]),
\end{equation}
remains conjectural (unless $\mathfrak{X}$ is smooth in which case there is no need to blow up and introduce the parameter $\gamma$, and then \eqref{sux} is given by \cite[Theorem 5.5]{Dy}).
\end{remark}
\section{Envelopes in nef and big classes}

In this section we prove
 Theorem \ref{thmenv}.  Assume we are in that setting, namely,
$(M, \omega)$ is compact K\"ahler (no boundary) and $\alpha$ is a closed real $(1,1)$ form such that $[\alpha]$ is nef  with $\int_M\alpha^n>0$. Let $u$ be the envelope
$$u(x)=\sup\{\vp(x)\ |\ \vp\in \textrm{PSH}(M,\alpha), \vp\leq 0\}.$$
Thanks to Demailly's regularization theorem \cite{Dem92,Dem94}, we have a K\"ahler current $T=\alpha+\ddbar\psi$ with analytic singularities along the Zariski closed set $E_{nK}(\alpha) \subset M$, such that $T\geq \delta\omega$ weakly on $M$, for some $\delta>0$.  We may assume without loss of generality that $\psi \le 0$.

We use the idea of Berman \cite{Be} (see also \cite{Lee} for a similar idea in the context of real Monge-Amp\`ere equations on domains in $\mathbb{R}^n$) who proved $C^{1,\gamma}_{\rm loc}(M\backslash E_{nK}(\alpha))$ regularity for all $0<\gamma<1$. The idea is the following: since $[\alpha]$ is nef, for every $\ve,\beta>0$ the result of Aubin and Yau \cite{A, Y} implies that we can find a smooth function $f_{\ve,\beta}$ satisfying
\begin{equation}\label{big and nef MA equ}
\alpha+\ve\omega+\ddbar f_{\ve,\beta}>0,\quad (\alpha+\ve\omega+\ddbar f_{\ve,\beta})^n=e^{\beta f_{\ve,\beta}}\omega^n,
\end{equation}
and as $\ve\to 0,\beta\to\infty$ the functions $f_{\ve,\beta}$ converge to the envelope $u$ (see the proof of Theorem \ref{thmenv} for details).

Our goal is then to prove $C^{1,1}_{\textrm{loc}}(M\backslash E_{nK}(\alpha))$ estimates for $f_{\ve, \beta}$,  independent of $\ve$ and $\beta$. First we have a lemma.
\begin{lemma}\label{le}
There exist uniform constants $B$, $C$ and $\beta_0$ independent of $\ve, \beta$ such that, for $\beta \ge \beta_0$,
\begin{enumerate}
\item[(i)]$\displaystyle{\psi-C\leq f_{\ve, \beta} \le C}$
\smallskip

\item[(ii)] $\displaystyle{\Delta_g f_{\ve, \beta} \le Ce^{-B\psi}}$
\smallskip

\item[(iii)] $\displaystyle{|\partial\fvebe|_{g}^{2}\leq Ce^{-B\psi}}$
\end{enumerate}
on $M\backslash E_{nK}(\alpha)$.
\end{lemma}
\begin{proof}
Parts (i) and (ii) are due to Berman \cite{Be}.  For convenience of the reader, we include the short proofs here.
For (i), the maximum principle immediately gives $\beta\sup_M f_{\ve,\beta}\leq C,$ independent of $\ve,\beta$. Then \cite[Theorem 4.1]{BEGZ} gives $f_{\beta,\ve}\geq V_{\ve}-C,$ with a uniform $C$, where $V_\ve$ is any fixed $(\alpha+\ve\omega)$-PSH function with minimal singularities. By definition of minimal singularities we get
\begin{equation}\label{min}
f_{\ve, \beta}\geq V_\ve-C\geq \psi-C,
\end{equation}
which establishes (i). Define
$$\ti{f}_{\ve, \beta} = f_{\ve, \beta} - \psi,$$
which is uniformly bounded from below.  To simplify notation, we will write $f$ for $\fvebe$ and $\ti{f}$ for $\tifvebe$.
Define $$\ti{\omega}=\alpha+\ve\omega+\ddbar f.$$
From (\ref{big and nef MA equ}),
$$\Ric(\ti{\omega})=\Ric(\omega)-\beta\ddbar f=\Ric(\omega)-\beta\ti{\omega}+\beta(\alpha+\ve\omega).$$
Using the well-known estimate of \cite{A,Y}, we have for a uniform $C$,
\[\begin{split}
\Delta_{\ti{g}}\log\tr{\omega}{\ti{\omega}}&\geq\frac{1}{\tr{\omega}{\ti{\omega}}}\left(-C(\tr{\ti{\omega}}{\omega})(\tr{\omega}{\ti{\omega}})-\tr{\omega}{\Ric(\ti{\omega})}\right)\\
&\geq -C\tr{\ti{\omega}}{\omega}+\beta-\frac{C\beta}{\tr{\omega}{\ti{\omega}}},
\end{split}\]
since $\beta$ may be assumed to be large. Now note that
\begin{equation} \label{fti}
-\ddbar \ti{f}=-\ti{\omega}+\alpha+\ve\omega+\ddbar\psi\geq-\ti{\omega}+\delta\omega,
\end{equation}
on $M\backslash E_{nK}(\alpha)$, and so on this set we have
$$\Delta_{\ti{g}}(\log\tr{\omega}{\ti{\omega}}-B\ti{f})\geq\frac{\beta}{2}-\frac{C\beta}{\tr{\omega}{\ti{\omega}}},$$
 for $\beta$ large as long as $B \ge C/\delta$. The maximum of this quantity cannot be achieved on $E_{nK}(\alpha)$, and so at a maximum point we get
$$\frac{\beta}{2}-\frac{C\beta}{\tr{\omega}{\ti{\omega}}}\leq 0,$$
namely $\tr{\omega}{\ti{\omega}}\leq C$.  From \eqref{min} we obtain at this point $\log\tr{\omega}{\ti{\omega}}-B\ti{f} \leq C$, and so this is true everywhere.  Since $f$ is uniformly bounded from above we have
$$\tr{\omega}{\ti{\omega}}\leq Ce^{-B\psi},$$
giving (ii).

For (iii), we consider the quantity
\begin{equation*}
Q=e^{h(\ti{f})}|\partial f|_{g}^{2},
\end{equation*}
where $h$ is defined by (cf. \cite{PS3})
\begin{equation*}\label{defnh}
h(s)=-Bs+\frac{1}{s-A+1},
\end{equation*}
for  $A=\inf_{M}\ti{f}$ (recall that $\ti{f}$ is uniformly bounded from below) and $B$ a uniform positive constant to be determined later.  Note that $h=h(\ti{f})$ is uniformly bounded from above, satisfies $h'<0$ and $h''>0$, and $h(\ti{f})$ differs from $-B\ti{f}$ at most by $1$.

It is sufficient to show that $Q$ is bounded from above by a uniform constant $C$, since then
$$| \partial f|^2_g \le C e^{-h} \le Ce^{B \ti{f}} \le  C' e^{-B\psi},$$
for $C'$ uniform, since $f$ is uniformly bounded from above.

Suppose $Q$ achieves a maximum at $x_0$, which must be a point of $M \setminus E_{nK}(\alpha)$.
We choose holomorphic normal coordinates centered at $x_{0}$ such that $(g_{i\ov{j}})$ is the identity and $(\tilde{g}_{i\ov{j}})$ is diagonal at $x_{0}$. At $x_{0}$, we have
\begin{equation}\label{big and nef gradient equ 1}
0\ge \Delta_{\tilde{g}}Q  = |\partial f |_{g}^{2}\Delta_{\tilde{g}}(e^{h})+e^{h}\Delta_{\tilde{g}}(|\partial f |_{g}^{2}) +2\mathrm{Re}\left(\tilde{g}^{i\overline{i}}\partial_{i}(e^{h})\partial_{\overline{i}}(|\partial f |_{g}^{2})\right).
\end{equation}
For the first term on the right hand side of (\ref{big and nef gradient equ 1}), using (\ref{fti}) and recalling that $h'<0$,
\begin{equation}\label{big and nef gradient equ 2}
\begin{split}
 |\partial f |_{g}^{2}\Delta_{\tilde{g}}(e^{h})\geq {} & e^{h}\left((h')^{2}+h''\right)|\partial f |_{g}^{2}|\partial\ti{f}|_{\tilde{g}}^{2}
+ne^{h}h'|\partial f |_{g}^{2} \\ {} & -e^{h}h'\delta|\partial f |_{g}^{2}\sum_{i}\tilde{g}^{i\overline{i}}.
\end{split}
\end{equation}
For the second term of (\ref{big and nef gradient equ 1}), writing $\alpha = \sqrt{-1} \alpha_{i\ov{j}} dz^i \wedge d\ov{z}^j$, we obtain for a uniform constant $C$,
\begin{equation}\label{big and nef gradient equ 3}
\begin{split}
e^{h}\Delta_{\tilde{g}}(|\partial f |_{g}^{2}) = {} & e^{h}\sum_{k}\tilde{g}^{i\overline{i}}(|f_{ik}|^{2}+|f_{i\overline{k}}|^{2}) + e^h \tilde{g}^{i\ov{i}} R_{i\ov{i}}^{\ \ \, \ov{\ell} k} \partial_k f \partial_{\ov{\ell}} f\\
{} & + 2 e^h \textrm{Re} \left( \sum_k \tilde{g}^{i\ov{i}} \partial_k (\tilde{g}_{i\ov{i}} - \alpha_{i\ov{i}} - \ve g_{i\ov{i}}) \partial_{\ov{k}} f \right) \\
\geq {} & e^{h}\sum_{k}\tilde{g}^{i\overline{i}}(|f_{ik}|^{2}+|f_{i\overline{k}}|^{2})-Ce^{h}|\partial f |_{g}^{2}\sum_{i}\tilde{g}^{i\overline{i}} \\
{} &  + 2 \beta e^h | \partial f|^2_g,
\end{split}
\end{equation}
where we applied $\partial_k$ to the equation (\ref{big and nef MA equ}).  We also assumed in (\ref{big and nef gradient equ 3}), without loss of generality, that $|\partial f|^2_g \ge 1$.

For the third term of (\ref{big and nef gradient equ 1}) we compute
\begin{equation}\label{big and nef gradient equ 4}
\begin{split}
    & 2\mathrm{Re}\left(\tilde{g}^{i\overline{i}}\partial_{i}(e^{h})\partial_{\overline{i}}(|\partial f|_{g}^{2})\right)\\
= ~ & 2\mathrm{Re}\left( \sum_k \tilde{g}^{i\overline{i}}e^{h}h' \ti{f}_{i} f_{\overline{k}} f_{\overline{i}k}\right)
+2\mathrm{Re}\left( \sum_k \tilde{g}^{i\overline{i}}e^{h}h' \ti{f}_{i} f_{k} f_{\overline{i}\overline{k}}\right).
\end{split}
\end{equation}
Next, by the Cauchy-Schwarz and Young's inequalities, we obtain
\begin{equation}\label{big and nef gradient equ 5}
\begin{split}
  &2\mathrm{Re}\left(\sum_k \tilde{g}^{i\overline{i}}e^{h}h' \ti{f}_i f_{\overline{k}} f_{\overline{i}k}\right)\\
=~&2\mathrm{Re}\left(\sum_k \tilde{g}^{i\overline{i}}e^{h}h' \ti{f}_i f_{\overline{k}}(\tilde{g}_{k\overline{i}}-\alpha_{k\overline{i}}-\ve g_{k\ov{i}})\right)\\
\geq ~ & 2e^{h}h' \mathrm{Re}\sum_i \ti{f}_i f_{\overline{i}}+ \frac{C}{\delta} e^{h}h'|\partial \ti{f}|_{\tilde{g}}^{2}
+\frac{\delta}{2} e^{h}h' |\partial f |_{g}^{2}\sum_{i}\tilde{g}^{i\overline{i}}\\
\geq ~ & 3e^{h}h'|\partial f |_{g}^{2}+e^{h}h'|\partial\psi|_{g}^{2}+ \frac{C}{\delta} e^{h}h'|\partial\ti{f}|_{\tilde{g}}^{2}
+\frac{\delta}{2} e^{h}h' |\partial f |_{g}^{2}\sum_{i}\tilde{g}^{i\overline{i}},
\end{split}
\end{equation}
where we used $h'<0$ and $\ti{f}=f-\psi$. On the other hand, we see that
\begin{equation}\label{big and nef gradient equ 6}
\begin{split}
&2\mathrm{Re}\left(\sum_k \tilde{g}^{i\overline{i}}e^{h}h' \ti{f}_i f_{k} f_{\overline{i}\overline{k}}\right)\\
\geq~& -e^{h}\sum_{k}\tilde{g}^{i\overline{i}}| f_{ik}|^{2}-e^{h}(h')^{2}|\partial f |_{g}^{2}|\partial \ti{f}|_{\tilde{g}}^{2}.
\end{split}
\end{equation}
Now the fact that $\psi$ has analytic singularities means that there is $\gamma\in\mathbb{R}_{>0}$ such that locally near every point of $M$ we can write $\psi=\gamma\log\sum_{j=1}^N |f_j|^2+v$, where $v$ is smooth and the $f_j$'s are local holomorphic functions (whose common zero locus locally cuts out $E_{nK}(\alpha)$). This implies that $e^{\frac{\psi}{\gamma}}$ is smooth on all of $M$ (in particular, its gradient squared is globally bounded), and so there is a uniform constant $C_1>0$ such that
\begin{equation} \label{psi}
| \partial \psi|^2_g \le C e^{-C_1 \psi},
\end{equation}
holds on $M\backslash E_{nK}(\alpha)$.
Combining (\ref{big and nef gradient equ 1}), (\ref{big and nef gradient equ 2}), (\ref{big and nef gradient equ 3}), (\ref{big and nef gradient equ 4}), (\ref{big and nef gradient equ 5}), (\ref{big and nef gradient equ 6}) and (\ref{psi}), at $x_{0}$, we have
\begin{equation}\label{big and nef gradient equ 7}
\begin{split}
0 & \geq  e^{h}h''|\partial f |_{g}^{2}|\partial\ti{f}|_{\tilde{g}}^{2}
+\left(-\frac{\delta}{2} h'-C_{0}\right)e^{h}|\partial f |_{g}^{2}\sum_{i}\tilde{g}^{i\overline{i}}\\
& \quad + e^{h}(C_0h' + 2\beta) |\partial f |_{g}^{2}+ C_0e^{h}h' e^{-C_1\psi}+\frac{C_0}{\delta} e^{h}h'|\partial\ti{f}|_{\tilde{g}}^{2},
\end{split}
\end{equation}
for uniform constants $C_{0}, C_1$.

We now choose the constant $B$:
$$B=\max\left(\frac{2C_{0}+2}{\delta},C_1\right).$$   For $\beta \ge C_0(B +1)\ge1$ we obtain from (\ref{big and nef gradient equ 7}) and the definition of $h$,
\begin{equation} \label{yacasi}
\begin{split}
0 & \geq  \frac{2}{(\ti{f}-A+1)^{3}}|\partial f |_{g}^{2}|\partial\ti{f}|_{\tilde{g}}^{2}
+|\partial f |_{g}^{2}\sum_{i}\tilde{g}^{i\overline{i}}\\
& \quad  + \beta |\partial f|_{g}^{2}-Ce^{-C_1\psi}- C |\partial\ti{f}|_{\tilde{g}}^{2},
\end{split}
\end{equation}
for a uniform constant $C$.  We may assume without loss of generality that at $x_0$,
$$|\partial f|^2_g \ge C (\ti{f} - A +1)^3,$$
since if not then  $Q$ is bounded from above at $x_0$.

Hence we obtain from (\ref{yacasi}),
$$| \partial f|^2_g \le \frac{Ce^{-C_1 \psi}}{\beta + \sum_i \tilde{g}^{i\ov{i}}}.$$
If at $x_0$ we have $f(x_0)>0$ then we obtain
$$|\partial f|^2_g \le Ce^{-C_1\psi} \le Ce^{C_1\ti{f}}\leq Ce^{B\ti{f}} \le C e^{-h},$$
for a uniform $C$
and we are done.

Otherwise, $f(x_0)\le 0$.  From the equation (\ref{big and nef MA equ}) and the arithmetic-geometric means inequality,
$$\sum_i \tilde{g}^{i\ov{i}} \ge n \left( \frac{\det g}{\det \ti{g}} \right)^{1/n} = n e^{-\frac{\beta}{n} f},$$
and hence, as long as $\beta \ge nB$, $$|\partial f|^2_g \le C e^{-C_1 \psi + \frac{\beta}{n} f} \le Ce^{-B\psi+B\ti{f}}=C e^{B \ti{f}} \le C e^{-h},$$ as required.
\end{proof}
\begin{remark}
The gradient estimate in part (iii) is analogous to the Phong-Sturm adaptation \cite{PS3} of B\l ocki's gradient estimate \cite{Bl2} to the degenerate case. B\l ocki's estimate was also used in \cite{DaR} along Berman's path of Monge-Amp\`ere equations, in the simpler case of a K\"ahler class. However the fact that we have both a degenerate class, as well as the parameter $\beta$, makes the details of our gradient estimate different from all of these references.
\end{remark}

We now prove a bound on the real Hessian of $\fvebe$.

\begin{lemma}\label{le2}
There exist uniform constants $B$ and $C$ such that
\begin{equation*}
|\nabla^{2}\fvebe|_{g}\leq Ce^{-B\psi}
\end{equation*}
on $M\backslash E_{nK}(\alpha)$, where $\nabla$ is the Levi-Civita connection of $\omega$.
\end{lemma}
\begin{proof}
The proof is very similar to that of Theorem \ref{thmap}, replacing $\varphi_{\gamma}$, $\varphi_{\gamma,\delta}$, $\omega^{(\gamma)}$ and $\log|s|_{h}^{2}$ by $\fvebe$, $\tifvebe$, $\alpha+\ve\omega$ and $\psi$ (note however that while $\omega^{(\gamma)}$ is K\"ahler, $\alpha+\ve\omega$ may not be).  For the reader's convenience we give here a brief sketch.

Again, we will write $f$ and $\ti{f}$ for $\fvebe$ and $\tifvebe$ respectively.  Let $B$ be a uniform constant at least as large as that of Lemma \ref{le}.
Consider the quantity
$$Q = \log \lambda_1( \nabla^2 f) + \rho(e^{B\psi} | \partial f |^2_g) - A\ti{f},$$
where, analogous to (\ref{defnrho}), the function $\rho$ is given by
\begin{equation} \label{rho2}
\rho(\tau) = - \frac{1}{2} \log (1+ \sup_M (e^{B\psi}|\partial f|^2_g)  - \tau).
\end{equation}

Assume that $Q$ achieves a maximum at $x_{0}$, which must be in $M \setminus E_{nK}(\alpha)$. As in Section \ref{sectiondeg},  choose normal holomorphic coordinates centered at $x_0$.  By the same calculation as (\ref{big and nef gradient equ 3}), and using the result of Lemma \ref{le}.(iii),
\[
\begin{split}
\Delta_{\tilde{g}} ( | \partial f|^2_g)
\ge {} & \sum_k \tilde{g}^{i\ov{i}} (| f_{ik}|^2+ | f_{i\ov{k}}|^2) - C e^{-B\psi} \sum_i \tilde{g}^{i\ov{i}}.
\end{split}
\]
Combining this with
$$\ddbar(e^{B\psi})\geq Be^{B\psi}\ddbar\psi\geq -Be^{B\psi}\alpha \geq -Ce^{B\psi}\omega,$$
with \eqref{psi} and the argument of \eqref{ggg}, we obtain
$$\Delta_{\tilde{g}}(e^{B\psi} | \partial f |^2_g)
\geq \frac{e^{B\psi}}{2} \sum_k \tilde{g}^{i\ov{i}} (| f_{ik}|^2+ | f_{i\ov{k}}|^2)   - C \sum_i \tilde{g}^{i\ov{i}},$$
after possibly increasing $B$ so that $e^{(B-C_1)\psi}|\de f|^2_g\leq C$.
Hence
\begin{equation}\label{equJ2}
\begin{split}
\Delta_{\tilde{g}}(\rho(e^{B\psi} | \partial f |^2_g))
\geq & ~ \frac{\rho'}{2}e^{B\psi} \sum_k \tilde{g}^{i\ov{i}} (| f_{ik}|^2+ | f_{i\ov{k}}|^2)
\\ & +\rho''\tilde{g}^{i\overline{i}}|\partial_{i}(e^{B\psi} | \partial f |^2_g)|^{2}
  -C\sum_{i}\tilde{g}^{i\overline{i}},
\end{split}
\end{equation}
which is the analog of Lemma \ref{lemmafo}.

We use the same perturbation argument as in the proof of Theorem \ref{thmap},  writing now $\lambda_1$ for the largest eigenvalue of the appropriate perturbed quantity.
Applying $V_{1}V_{1}$ to the logarithm of (\ref{big and nef MA equ}), at $x_{0}$, we see that
\begin{equation}\label{big and nef second equ 2}
\begin{split}
\tilde{g}^{i\overline{i}}(V_{1}V_{1}(\tilde{g}_{i\overline{i}})) & = g^{p\overline{p}}g^{q\overline{q}}|V_{1}(\tilde{g}_{p\overline{q}})|^{2}
+V_{1}V_{1}(\log\det g)+V_{1}V_{1}(\beta f)\\
& \geq g^{p\overline{p}}g^{q\overline{q}}|V_{1}(\tilde{g}_{p\overline{q}})|^{2}-C,
\end{split}
\end{equation}
where we used $V_{1}V_{1}(f)=\lambda_{1}$. As in Lemma \ref{lemmalbl1}, at $x_0$ this gives
$$\Delta_{\ti{g}} \lambda_1 \ge 2\sum_{\alpha>1} \tilde{g}^{i\ov{i}} \frac{| \partial_i (f_{V_{\alpha} V_1}) |^2}{\lambda_1-\lambda_{\alpha}} + \tilde{g}^{p\ov{p}} \tilde{g}^{q\ov{q}} |V_1 (\tilde{g}_{p\ov{q}})|^2 - C\lambda_1 \sum_i \tilde{g}^{i\ov{i}}.$$
Using (\ref{fti}) to obtain $-\Delta_{\ti{g}} (A \ti{f}) \ge A\delta \sum_i \tilde{g}^{i\ov{i}} -An$, we have
\begin{equation} \label{LhatQ4d}
\begin{split}
0
\ge {} & 2 \sum_{\alpha >1}  \frac{\tilde{g}^{i\ov{i}} |\partial_i (f_{V_{\alpha} V_1})|^2}{\lambda_1(\lambda_1-\lambda_{\alpha})} + \frac{\tilde{g}^{p\ov{p}} \tilde{g}^{q\ov{q}} | V_1(\tilde{g}_{p\ov{q}})|^2}{\lambda_1} - \frac{\tilde{g}^{i\ov{i}} | \partial_i (f_{V_1 V_1})|^2}{\lambda_1^2} \\
{} & + \frac{\rho'}{2}e^{B\psi} \sum_k \tilde{g}^{i\ov{i}} (| f_{ik}|^2+ | f_{i\ov{k}}|^2) + \rho''\tilde{g}^{i\overline{i}}|\partial_{i}(e^{B\psi} | \partial f |^2_g)|^{2} \\ {} &+ \left(A\delta-C\right) \sum_i \tilde{g}^{i\ov{i}}
 - An,
\end{split}
\end{equation}
which is the analog of Lemma \ref{lemmalbl1}.
Combining these and the rest of the arguments of Theorem \ref{thmap} (there is only a small change in the proof of the analog of Lemma \ref{lemmauno} due to the fact mentioned earlier that $\alpha+\ve\omega$ need not be K\"ahler, but this does not really affect the arguments), we obtain $\lambda_{1}\leq Ce^{-B\psi}$ at $x_{0}$, as required.
\end{proof}

It is now straightforward to complete the proof of Theorem \ref{thmenv}, given the results of Berman \cite{Be}.

\begin{proof}[Proof of Theorem \ref{thmenv}]
Since the class $[\alpha]+\ve[\omega]$ is K\"ahler for all $\ve>0$, it follows from \cite[Proposition 2.4]{Be} that for every fixed $\ve>0$ as we let $\beta\to+\infty$ we have
$$f_{\ve,\beta}\to u_\ve,$$
uniformly on $X$, where $u_\ve$ is the envelope given by
\begin{equation}\label{env2}
u_\ve(x)=\sup\{\vp(x)\ |\ \vp\in \textrm{PSH}(M,\alpha+\ve\omega), \vp\leq 0\}.
\end{equation}
From the definition it is clear that the functions $u_\ve$ are decreasing as $\ve$ decreases to zero, and it is easy to see that their pointwise decreasing limit is $u$ (see \cite[Lemma 5.2]{BEGZ}). Since Lemmas \ref{le} and \ref{le2} give uniform $C^{1,1}$ bounds for $f_{\ve,\beta}$ on compact subsets away from $E_{nK}(\alpha)$, it follows easily that $u$ is  $C^{1,1}_{\rm loc}(M\backslash E_{nK}(\alpha))$, as desired.

Lastly we prove \eqref{volform}, following the same lines as \cite{Be3,BD,To}. Away from $E_{nK}(\alpha)$ we can define the Monge-Amp\`ere operator $(\alpha+\ddbar u)^n$ (either in the sense of Bedford-Taylor, or pointwise a.e.), and it is classical (see e.g. \cite[Proposition 3.1]{Be3}) that this vanishes outside the contact set $\{u=0\}$, and so
$$\int_{M\backslash E_{nK}(\alpha)}(\alpha+\ddbar u)^n=\int_{(M\backslash E_{nK}(\alpha))\cap\{u=0\}}(\alpha+\ddbar u)^n.$$
 Since $u$ has minimal singularities, it follows from \cite{BEGZ} that
$$\int_M \langle (\alpha+\ddbar u)^n\rangle=\int_M \alpha^n,$$
where $\langle \cdot\rangle$ denotes the non-pluripolar Monge-Amp\`ere product. Since $u\in C^{1,1}_{\rm loc}(M\backslash E_{nK}(\alpha))$, it follows that the non-pluripolar product equals the extension by zero of $(\alpha+\ddbar u)^n$ from $M\backslash E_{nK}(\alpha)$ to $M$, and so we get
$$\int_{M\backslash E_{nK}(\alpha)}(\alpha+\ddbar u)^n=\int_M \alpha^n.$$
But since $u\in C^{1,1}_{\rm loc}(M\backslash E_{nK}(\alpha))$, a simple measure-theoretic argument (see e.g. \cite[p.2]{To}) shows that $\nabla^2 u=0$ a.e. on the set $(M\backslash E_{nK}(\alpha))\cap\{u=0\}$, so that on this set we have $\alpha+\ddbar u=\alpha$ a.e. (with respect to the Riemannian measure, and therefore also with respect to $(\alpha+\ddbar u)^n$ since this measure is absolutely continuous with respect to the Riemannian measure away from $E_{nK}(\alpha)$), which gives
$$\int_{(M\backslash E_{nK}(\alpha))\cap\{u=0\}}(\alpha+\ddbar u)^n=\int_{(M\backslash E_{nK}(\alpha))\cap\{u=0\}}\alpha^n=\int_{\{u=0\}}\alpha^n,$$
and putting these together we obtain \eqref{volform}.
\end{proof}
\begin{remark}
The proof of Theorem \ref{thmenv} (especially Lemma \ref{le2}) shows that the conclusion of the Main Theorem 1.2 of \cite{Be} is now improved to $C^{1,1}_{\rm loc}$ convergence on the complement of the non-K\"ahler locus.
\end{remark}

\end{document}